\documentclass{article}
\usepackage{amssymb,amsmath,amsfonts,amsthm,latexsym}
\usepackage{multicol}

\usepackage[utf8]{inputenc}

\newtheorem*{theo*}{Theorem}
\newtheorem{theo}{Theorem}
\newtheorem{coro}[theo]{Corollary}
\newtheorem{prop}[theo]{Proposition}
\newtheorem{lemm}[theo]{Lemma}
\newtheorem*{lemm*}{Lemma}

\newtheorem*{clai*}{Claim}

\theoremstyle{definition}
\newtheorem{rema}{Remark}
\newtheorem{defi}[rema]{Definition}

 \def\NN{{\mathbb N}}  \def\PP{{\mathbb P}}
 \def\RR{{\mathbb R}}

\def\La{\Lambda}

\def\cA{{\cal A}}    
\def\cB{{\cal B}}   \def\cN{{\cal N}} \def\cT{{\cal T}}
\def\cC{{\cal C}}    \def\cU{{\cal U}}
    \def\cV{{\cal V}}
\def\cE{{\cal E}}    \def\cW{{\cal W}}
    \def\cX{{\cal X}}

\def\set#1{\left\{\, #1 \,\right\}}

\def\abs #1{\vert \,#1\, \vert\,}
\def\norm #1{\Vert \,#1\, \Vert\,}

\newcommand{\vf}{\mathcal{X}^1(M)}

\newcommand{\flow}[1][t]{\phi_{#1}}

\newcommand{\pflow}[1][t]{\psi_{#1}}
\newcommand{\spflow}[1][t]{\psi^*_{#1}}
\newcommand{\sing}[1][X]{Sing(#1)}

\title{ Singular robustly chain transitive sets are singular volume partial hyperbolic}
\author{Adriana da Luz}

\begin{document}

\maketitle
\begin{abstract}
 For diffeomorphisms or for non-singular
flows, there are many results
relating properties  persistent  under $\cC^1$ perturbations and a global structures for  the dynamics ( such as hyperbolicity,
partial hyperbolicity, dominated splitting).
However, a dif\'iculty appears when a robust property of a
ow holds on a set containing
recurrent orbits accumulating a singular point. 

In \cite{BdL} with Christan Bonatti we propose a
a general procedure for adapting the usual hyperbolic structures to the singularities.

Using this tool, we recover the results in  \cite{BDP} for flows, showing that robustly chain transitive sets have a weak form of hyperbolicity. allowing us to conclude as well the kind of hyperbolicity carried by the examples in \cite{BLY} (a robust chaintransitive singular attractor with periodic orbits of different indexes). 

Along with the results in  \cite{BdL},this shows that the way we propose to interpret the effect of singularities, has the potential to adapt to other settings in which there is coexistence of singularities and regular orbits with the goal of reobtaining the results that we already know for diffeomorphisms.

\end{abstract}



\section{Introduction}
\subsection{General setting and historical presentation}
While studying nature, the parameters of the systems in question will be determined by measurements that, for the fact of being made by humans, will introduce some error. It is natural to ask how resistant  our conclusions are to this inevitable constrain.
A \emph{robust properties} is a property that is
impossible to break  by small perturbations of a system; in other words,
a dynamical property is robust if its holds on a (non-empty)  $\cC^1$ open set of diffeomorphisms or flows.

If the goal is then to study systems that present robust properties, we need to ask which robust properties are we interested in.

In this regard we might not be interested in all the orbits in our manifold, but rather in a subset, that can be shown to contain the relevant dynamical information: that is the chain recurrence classes.
For this we consider pseudo orbits, that is a generalization of an orbit   but at every iterate we allow an $\epsilon$ mistake or jump, then follow this new orbit (a precise definition is presented in the preliminaries section).
In this sense, the chain recurrence set is the set of recurrent pseudo orbits, and the chain recurrence classes are compact invariant sets of points that can be connected by pseudo orbits for any $\epsilon.$

 It is shown by Conley in \cite{Co} that this chain classes play the role of fundamental pieces of the dynamics, and the rest of the orbits, simply go from one of this pieces to the other.

Recall that a maximal invariant set is the intersection of the iterates of an open set. And a  chain recurrence class $C$  is said to be
 robustly  transitive  if  there is a neighborhood $U\subset M$ and  a neighborhood of $f$, $\cU$  such that the maximal invariant set in $U$  is a unique chain class $C$ and has a dense orbit for any $g\in\cU$.
The definition of robustly chain transitive set is a generalization of this notion since it is equivalent  to ask only that there is a neighborhood $U\subset M$ and  a neighborhood of $f$, $\cU$  such that the maximal invariant set in $U$ is a unique chain class  $C$ for any  $g\in\cU$.

Reasoning as in the beginning of this introduction, it is also natural to ask ourselves:  What is the situation when there are no robust properties?

In a series of papers first by Ma\~{n}\'{e} \cite{Ma}, for surfaces and then \cite{DPU} for 3 manifolds and culminating with \cite{BDP} for the more general result, a dichotomy  is presented between some hyperbolic like property and  the appearance by perturbation of infinitely many sinks and sources.  This shows that a Chain recurrence class might split under perturbation into infinitely many classes and there is no  topological dynamical property that was preserved in this process. 

The weak hyperbolic structure that forbids the appearance by perturbation of infinitely many sinks and sources was introduced by Ma\~ n\'e and Liao
and is called \emph{dominated splitting}:
 \begin{defi} Let $f\colon M\to M$ be a diffeomorphism of a Riemannian manifold $M$  and $K\subset M$ a compact invariant set of $f$, that is $f(K)=K$.
 A splitting $T_x M=E(x)\oplus F(x)$, for $x\in K$, is called dominated if
 \begin{itemize}
 \item $dim(E(x))$ is independent of $x\in K$ and this dimension is called the $s$-index of the splitting;
  \item it is $Df$-invariant:  $E(f(x))=Df(E(x))$ and $F(f(x))=Df(F(x))$ for every $x\in K$;
  \item there is $n>0$ so that for every $x$ in $K$ and every  unit vectors $u\in E(x)$ and $v\in F(x)$ one has
  $$\|Df^n(u)\|\leq \frac 12\|Df^n(v)\|.$$
 \end{itemize}
One denotes $TM|_K=E\oplus_{_<}F$ the dominated splitting.
\end{defi}

The results in  \cite{BDP}  is as follows

\begin{theo}(theorem 1 in\cite{BDP})\label{BDP1}
 Let $P$ be a hyperbolic saddle of a diffeomorphism $f$ defined
on a compact manifold M. Then
\begin{itemize}
  \item either the homoclinic class $H(P, f)$ of $P$ admits a dominated splitting,
  \item or given any neighborhood $U$ of $H(P, f)$ and any $k \in \NN$ there is $g$ arbitrarily $\cC^1$-close to $f$ having $k$ sources or sinks,
whose orbits are included in $U$.
\end{itemize}

\end{theo}

\begin{theo}(theorem 2 in\cite{BDP})\label{BDP2}
Every $C^1$-robustly transitive set (or diffeomorphism) admits
a dominated splitting.
\end{theo}
The conclusion also holds for an open and dense subset of vector fields and  robustly chain transitive sets.  This shows us that asking that the chain recurrence classes are preserved by perturbation induces a constraint in the dynamics of the vector field, that is, there must be a hyperbolic like property in the tangent space.

A generalization of this result implying the version for flows was given by \cite{BGV}.
Also in \cite{BDP} it is shown that a robustly chain transitive sets must have a weak hyperbolic structure, that is :
\begin{theo}(theorem 4 \cite{BDP})\label{BDP3}
Let $\La_f (U)$ be a $C^1$-robustly transitive set and $E^1\oplus\dots \oplus E^k$,
 be its finest dominated splitting. Then there exists $n \in \NN$ such
that $Df^n$ contracts uniformly the volume in $E^1$ and expands uniformly the
volume in $E^k$.
\end{theo}
A set $K$ is \emph{volume partial hyperbolic} if  there is a dominated splitting $E^{cs}\oplus E^c\oplus E^{cu}$ so that the volume in $E^{cs}$ is uniformly
  contracted and the volume in
  $E^{cu}$
  is uniformly expanded.

The aim of this paper is to generalize this results to flows with singularities.
But the fact that in fifteen years this has not been done, gives the idea that there might be additional difficulties to deal with in this scenario.
Firstly, a direct generalization of theorem \ref{BDP1} is not possible even for non singular flows,  there are transitive sets that do not have a dominated splitting for the tangent space (for instance take the suspension of the example in \cite{BV})
many hyperbolic structures for flows are not expressed in terms of the differential of the flow, 
but on its transverse structure (called the
\emph{linear Poincar\'e flow}). For non singular flows it has been shone that the results in \ref{BDP1} and  \ref{BDP2} can be generalized (in \cite{V}, \cite{BGV} and \cite{D}).
 
 However the linear Poincar\'e flow is only defined far from the singularities, and therefore it cannot be used directly for understanding our problem. And there are in fact open sets of vector fields having robustly chain transitive singular chain recurrence classes.
 \begin{itemize}
 \item There are many  examples of singular robustly chain transitive singular chain recurrence classes, with the extra property of having all periodic orbits robustly hyperbolic that we called star flows (For instance the Lorenz attractor or very many others). In this cases much is known of their hyperbolic structures. see for instance \cite{MPP}, \cite{GLW}, \cite{GSW} and \cite{BdL}.
\item In \cite{BLY}  The authors present an example of a vector field  with a chain transitive attractor in dimension 4. This attractor  has  periodic orbits of different stable index and singularities. This example is not a star flow.
\item The example just mentioned can be multiplied by a strong expansion for it not to be an attractor any more (now the example would be in a manifold of  dimension 5).
\end{itemize}

For the last two items the kind of weak hiperbolicity they carry or not was yet unknown. However it is also evident that there are not that many examples of nontrivial singular chain recurrence clases out of the star flow setting. 

In \cite{GLW},  the authors  define the notion of \emph{extended linear Poincar\'e flow} defined on some sort of blow-up of the singularities. The extension of the  chain recurrence class  $\La$  in $U$  with the blow-up, will be noted $\widetilde{\La}$.

The definition of this blow up of the singularity relies on information of the perturbations of the vector field, so that this set varies upper semi continuously with the vector field  in the $\cC^1$ topology

But in \cite{BdL} we propose a bigger set that (as opposed to the first case) does not depend on knowing any information from the neighbor vector fields. This set will be called \emph{extended maximal invariant} and noted $B(X,U)$. We belive a dominated splitting should be a property that one can check without information of the surrounding vector fields.
%But in this problem

Our notion of \emph{singular volume hyperbolicity} will be expressed as  the volume  hyperbolicity of a well chosen reparametrization of this extended linear Poincar\'e flow
over  $B(X,U)$ . But it is enough to do it over  $\widetilde{\La}$, since in \cite{BdL} it is shown that they both carry the same hyperbolic properties.

\subsection{The singular volume partial hyperbolicity}

In this section we take a closer look at weaker forms of hyperbolicity and their relation with the persistence of the dynamical properties.

Following the proofs in \cite{BDP} we show the following
\begin{prop}\label{t.periodic}  Let  $\cU\subset\cX^1(M)$ be a  $C^1$-open set such that, for every $X\in \cU$ there is an open set $U$ of $M$ such that the maximal invariant set in $U$ is an isolated chain recurrence class $C$. Then $\widetilde{\La}$ has a uniform finest  dominated splitting for the linear Poincar\'{e} flow:
$$\cN_L=\cN_L^1\oplus\dots\oplus \cN_L^n\,.$$
 Each of this periodic orbits is volume partial hyperbolic for the tangent space.
\end{prop}
We divide the singular set $Sing(X)\cap U$ in subsets sets:

\begin{itemize}\item the set $S_{Ec}$ of singular points whose escaping stable space has  dimension smaller than $\cN_L^1$,
\item the set $S_{E}$ of singular points whose escaping stable space has  dimension bigger or equal  than $\cN_L^1$,
\item the set $S_{Fc}$ of singular points whose escaping unstable space has  dimension smaller  than $\cN_L^n$,
\item the set $S_{F}$ of singular points whose escaping unstable space has  dimension bigger or equal  than $\cN_L^n$.
\end{itemize}
We want:
\begin{itemize} \item to reparametrize  the cocycle $\psi^t_{\cN}$ in restriction to $\cN_L^1$ by
the expansion  in the direction $L$ if and only if the line $L$ is based at a point close to  $S_{Ec}$;
\item  to reparametrize  the cocycle $\psi^t_{\cN}$ in restriction to $\cN_L^n$ by
the expansion  in the direction $L$ if and only if the line $L$ is based at a point close to  $S_{Fc}$.
\end{itemize}
For this we use again the   cocycle \emph{center-stable cocycle}  $\{h^t_{Ec}\}_{t\in\RR}$,so that:
\begin{itemize}
\item $h_{Ec}^t(L)$ and $\frac 1{h_{Ec}^t}(L)$ are uniformly bounded (independently of $t$), if $L$ is based on a point  $x$ so that $x$ and $\phi^t(x)$ are out of a small neighborhood of $S_{Ec}$,
where $\phi^t$ denotes the flow of $X$;
 \item $h_{Ec}^t(L)$ is in a bounded ratio with the expansion of $\phi^t$ in the direction $L$, if $L$ is based at a point $x$ so that $x$ and $\phi^t(x)$ are out of a small neighborhood of $S_E$.
 \item $h_{Ec}^t$ depends continuously on $X$.
\end{itemize}
Analogously we get the notion of \emph{center-unstable cocycles} $\{h_{Fc}^t\}$ by exchanging the roles of $S_{Ec}$ and $S_{Fc}$ in the properties above.

Now similarly to the  multisingular hyperbolicity case, we define the singular volume partial hyperbolicity.

\begin{defi} Let $X$ be a $C^1$-vector field on a closed manifold $M$.  Let $U$ be a compact set. We say that $X$ is \emph{singular volume partial hyperbolic in $U$} if:
\begin{itemize}
 \item the extended linear Poincar\'e flow admits a finest dominated splitting $\cN_L=\cN_L^1\oplus\dots\oplus \cN_L^n\,.$ over the pre extended maximal invariant set $\widetilde{\La}$.
 %\item every singular point in $U$ is hyperbolic.
 \item the set of singular points in $U$ is the  union  of $$S_{Ec}\cup S_{Fc}\cup S_E\cup S_F\,,$$
defined above.
 \item the reparametrized linear Poincar\'e flow $h_{Ec}^t \psi^t_{\cN}$ contracts volume on $\cN_L^1$, where $h_{Ec}^t$ is a center-stable cocycle,
 \item the reparametrized linear Poincar\'e flow $h_{Fc}^t \psi^t_{\cN}$ contracts volume on $ \cN_L^n$, where $h_{Fc}^t$ is a center-unstable cocycle,
\end{itemize}

\end{defi}

\subsection{Robustly chain transitive singular sets}
We now state the main result.
\begin{theo}\label{t.BDP}  Let  $\cU\subset\cX^1(M)$ be a  $C^1$-open set such that, for every $X\in \cU$ there is an open set $U$ of $M$ such that the  chain recurrence class $C$ is robustly chain transitive in $U$.  Then $X$ is singular volume partial hyperbolic in $U$.
\end{theo}

\section{Basic definitions and preliminaries}
\subsection{Chain recurrent classes and filtrating neighborhoods }

The following notions and theorems are due to Conley \cite{Co} and they can be found in several other references (for example \cite{AN}).
\begin{itemize}
  \item We say that pair of sequences $\set{x_i}_{0\leq i\leq k}$ and  $\set{t_i}_{0\leq i\leq k-1}$, $k\geq 1$,  are an \emph{ $\varepsilon$-pseudo orbit from $x_0$ to $x_k$} for a flow $\phi$,
  if for every $0\leq i \leq k-1$ one has
  $$ t_i\geq 1 \mbox{ and }d(x_{i+1},\phi^{t_i}(x_i))<\varepsilon.$$

  \item A compact invariant set $\Lambda$ is called \emph{chain transitive} if    for any $\varepsilon > 0$ and  for any $x, y \in\Lambda$
there is an $\varepsilon$-pseudo orbit from $x$ to $y$.
  \item We say that $x, y \in M$ are chain related if,  for every $\varepsilon>0$, there are $\varepsilon$-pseudo orbits form $x$ to $y$ and from $y$ to $x$. This is an equivalence relation.
\item  We say that $x\in M$ is \emph{chain recurrent} if  for every $\varepsilon>0$, there is an $\varepsilon$-pseudo orbit from $x$ to $x$. We call the set of chain recurrent points, the\emph{ Chain recurrent set}
and we note it $\mathfrak{R}(M)$. The equivalent classes of this equivalence relation are called \emph{ chain recurrence classes}.
\end{itemize}

\begin{defi}
\begin{itemize} \item An \emph{attracting region} (also called \emph{trapping region} by some authors) is a compact set $U$ so that $\phi^t(U)$ is contained in the interior
of $U$ for every $t>0$. The maximal invariant set in an attracting region is called an \emph{attracting  set}.  A repelling region is an attracting region for $-X$, and the maximal invariant set is called a repeller.

\item A \emph{filtrating region} is the intersection of an attracting region with a repelling region.

\item Let $C$ be a chain recurrent class of $M$ for the flow $\phi$.
A \emph{filtrating neighborhood } of $C$ is a (compact) neighborhood which is a filtrating region.
%\begin{itemize}
%\item An \emph{attracting set} for $\phi$ is a compact invariant set $A$ such that there is an open neighborhood  $U$ containing $A$ such that
%\begin{itemize}
%  \item  $ \phi^t(\overline{U})\subset U $  for all $t>0$,
%  \item  and $A=\bigcap_{t\in\RR^{+}}\phi(\overline{U})$
%\end{itemize}
%We call $U$ an \emph{attracting neighborhood}.
%\item %For an attracting set $A$, we define the dual repeller of $A$ to the set
%$R=\bigcap_{t\in\RR^{-}}\phi(\overline{M\setminus \overline{U}})$.
%A \emph{repeller}  $R$ is an attractor for the reverse flow. The attracting neighborhoods of $R$ for the reverse flow are called \emph{repelling neighborhoods}.
\end{itemize}
\end{defi}

%\begin{defi}
%Let $C$ be a Chain recurrent class of $M$ for the flow $\phi$.
%We say that $U$ is a \emph{filtrating neighborhood } if there exist an attracting neighborhood $V$ and a repelling neighborhood $W$ such that $$U= V\cap \,W$$
%\end{defi}
The following is a corollary of the fundamental theorem of dynamical systems  \cite{Co}.
\begin{coro}\cite{Co}
Let $\phi$  be a $C^1$-vector field on a compact manifold $M$. Every chain class $C$  of $X$ admits a  basis of filtrating neighborhoods, that is, every neighborhood
of $C$ contains a filtrating neighborhood of $C$.
\end{coro}

\begin{defi}
Let $C$ be a Chain recurrent class of $M$ for the vector field $X$. Let $C$ be such that there is a filtrating neighborhood $U$ such that $C$ is the maximal invariant set in $U$.
We say that $C$  is \emph{robustly chain transitive} if there is a $C^1$ neighborhood of $X$ called $\mathcal{U}$ such that for every $Y\in\mathcal{U}$, the maximal invariant set for $Y$ ($C_Y$) in $U$ is a unique  chain class.
\end{defi}

\begin{defi}
Let $C$ be a robustly chain transitive class of $M$ for the vector field $X$.
We say that $C$  is robustly transitive if there is a $C^1$ neighborhood of $X$ called $\mathcal{U}$ such that for every $Y\in\mathcal{U}$, there is an orbit for $Y$ which is dense in $C_Y$.
\end{defi}

%%%%%%%%%%%%%%%%%%%%%%%%%%%%%%%%%%%%
\subsection{Linear cocycle}
%%%%%%%%%%%%%%%%%%%%%%%%%%%%%%%%%%%%%
Let  $\phi=\{\phi^t\}_{t\in\RR}$ be a topological flow on a compact metric space  $K$.
A \emph{linear cocycle over   $(K,\phi)$}   is a continuous map   $A^t\colon E\times \RR\to E$
defined by $$A^t(x, v) = (\phi^t(x), A_t(x)v)\,,$$ where
\begin{itemize}
\item  $\pi\colon E\to K$ is a $d$ dimensional  linear bundle over $K$;
\item  $A_t:(x,t)\in K\times \RR \mapsto GL(E_x,E_{\phi^t(x)})$ is a continuous map  that satisfies the \emph{cocycle relation } :
 $$A_{t+s}(x)=A_t(\phi^s(x))A_s(x),\quad  \mbox{ for any } x\in K \mbox{ and }t,s\in\RR  $$

\end{itemize}

Note that $\cA=\{A^t\}_{t\in\RR}$ is a flow on the space $E$ which projects on $\phi^t$.
$$\begin{array}[c]{ccc}
E &\stackrel{A^t}{\longrightarrow}&E\\
\downarrow&&\downarrow\\
K&\stackrel{\phi^t}{\longrightarrow}&K
\end{array}$$

If $\Lambda\subset K$ is a $\phi$-invariant subset,  then $\pi^{-1}(\Lambda)\subset E$ is $\cA$-invariant, and  we  call
\emph{the restriction of $\cA$ to $\Lambda$}  the restriction of $\{A^t\}$ to $\pi^{-1}(\Lambda)$.

%%%%%%%%%%%%%%%%%%%%%%%%%%%%%%%%%%%%%%%%%%%%%%%%%%%%%%%%%%%%%%%%%%%%
 \subsection{Hyperbolic structures and  dominated splitting on linear cocycles}
%%%%%%%%%%%%%%%%%%%%%%%%%%%%%%%%%%%%%%%%%%%%%%%%%%%%%%%%%%%%%%%%%%%%%%%

In this section we give a rough presentation of some of the hyperbolic structures over cocycles.
 \begin{defi}
 Let $\phi$ be a topological flow on a compact metric space $\La$. We consider a vector bundle  $\pi\colon E\to \La$ and
   a linear cocycle $\cA=\{ A^t\}$ over $(\La,X)$.

 We say that $\cA$ admits a \emph{dominated splitting over $\Lambda$} if
 \begin{itemize}\item there exists a splitting $E=E^1\oplus\dots\oplus E^k$ over $\lambda$ into $k$ sub bundles
 \item  The dimension of the sub bundles is constant, i.e. $dim(E^i_x)=dim(E^i_y)$ for all $x,y\in\Lambda$ and $i\in \set{1\dots k}$,
     \item The splitting is invariant, i.e. $A^t(x)(E^i_x)=E^i_{\phi^t(x)}$ for all $i\in \{1\dots k\}$,

\item there exists a $t>0$ such that for every $x\in \Lambda$ and any pair of non vanishing vectors $v\in E^i_x$ and $u\in E^j_x$, $i<j$ one has
 \begin{equation}\label{e.dom}
 \frac{\norm{A^t(u)}}{\norm u}\leq \frac{1}{2}\frac{\norm{ A^t(v)}}{\norm v}
 \end{equation}

 We denote $E^1\oplus_{_\prec}\dots\oplus_{_\prec} E^k$.% or $E^1\oplus_{{_\prec}_t}\dots\oplus_{{_\prec}_t} E^k$ if one wants to emphasis the role of $t$: in that case one says that
 the splitting is \emph{$t$-dominated}.
\end{itemize}
 \end{defi}

 A classical result (see for instance \cite[Appendix B]{BDV})  asserts that the bundles of a dominated splitting varies continuously with the vector field in the $\cC^1$ topology.
 A given cocycle may admit several dominated
  splittings.  However, the dominated splitting is unique if one prescribes the dimensions $dim(E^i)$.

  We can consider in a metric spaces $K$, an invariant sub spaces $\La$ of $K$ that is not compact. In this case we would ask for the norm of   $\cA$ to be bounded. Note that the dominated splitting defined as above is uniform with respect to the point. This is particularly important when we consider a dominated splitting over a set that is not compact.

 One says that one of the bundle $E^i$ is \emph{volume contracting} (resp. \emph{expanding}) if there is $t>0$ so that
one has $$Det(J(A^t\mid_{E^i}))<\frac 12$$
 (resp. $Det(J(A^{-t}\mid_{E^i}))<\frac 12$.% In both cases one says that $E^i$ is \emph{hyperbolic}.

% Notice that is $E^j$ is contracting (resp. expanding) then the same holds for any $E^i$, $i<j$  (reps. $j<i$) as a consequence of the domination.

\begin{defi}
 We say that the linear cocycle  $\cA$ is \emph{volume partial hyperbolic over $\Lambda$} if
  there is a finest  dominated splitting $E=E^1\oplus_{_\prec} \dots \oplus E^l$ over $\Lambda$ is such that the   extremal bundles  $E^1$ and resp. $E^l$ volume contracting/expanding
 \end{defi}

 \subsection{Linear Poincar\'e flow}
%%%%%%%%%%%%%%%%%%%%%%%%%%%%%%%%%%%%%%%%%

 Let $X$ be a $C^1$ vector field on a compact manifold $M$.  We denote by $\phi^t$ the flow of $X$.

 \begin{defi} The \emph{normal bundle} of $X$ is the vector bundle $N_X $ over $M\setminus Sing(X)$ defined as follows: the fiber $N_X(x)$ of $x\in M\setminus Sing(X)$ is
 the quotient space of $T_xM$ by the vector line $\RR.X(x)$.
 \end{defi}
 Note that, if $M$ is endowed with a Riemannian metric, then $N_X(x)$ is canonically identified with the orthogonal space of $X(x)$:
 $$N_X=\{(x,v)\in TM, v\perp X(x)\} $$

Consider $x\in M\setminus Sing(M)$ and $t\in \RR$.  Thus $D\phi^t(x):T_xM\to T_{\phi^t(x)}M$ is a linear automorphism mapping $X(x)$ onto $X(\phi^t(x))$. Therefore
$D\phi^t(x)$ passes to the quotient as an linear automorphism $\psi^t(x)\colon N_X(x)\to N_X(\phi^t(x))$:

$$\begin{array}[c]{ccc}
T_xM&\stackrel{D\phi^t}{\longrightarrow}&T_{\phi^t(x)}M\\
\downarrow&&\downarrow\\
N_X(x)&\stackrel{\psi^t}{\longrightarrow}&N_X(\phi^t(x))
\end{array}$$
where the vertical arrow are the canonical projection of the tangent space to the normal space parallel to $X$.

 \begin{prop}\label{p.lpf} Let $X$ be a $C^1$ vector field on a manifold and $\La$ be a compact invariant set of $X$.  Assume that $\La$ does not contained any singularity of $X$.
 Then $\La$ is hyperbolic if and only if the linear Poincar\'e flow over $\La$ is hyperbolic.
 \end{prop}

 Notice that the notion of dominated splitting for non-singular flows is sometimes better expressed in term of Linear Poincar\'e flow: for instance,
 the linear Poincar\'e flow of a robustly transitive vector field
 always admits a dominated splitting, when the flow by itself may not admit any dominated splitting (see for instance \cite{V} and \cite{D}).

 %%%%%%%%%%%%%%%%%%%%%%%%%%%%%%%%%%%%%%%%%%%%%%%%%%%%
 \subsection{Extended linear Poincar\'e flow}\label{ss.Poincare}
 %%%%%%%%%%%%%%%%%%%%%%%%%%%%%%%%%%%%%%%%%%%%%%%%%%%%

 We are dealing with singular flows and the linear Poincar\'e flow is not defined on the singularity of the vector field $X$.
 However we can include the linear Poincar\'e flow  in a
 flow, called \emph{extended linear Poincar\'e flow} defined in \cite{GLW}, on a larger set.

 This flow will be a linear co-cycle define on some linear bundle over a manifold, that we define now.

 \begin{defi}Let $M$ be a  manifold of dimension $d$.
 \begin{itemize}
 \item We call \emph{the projective tangent bundle of $M$}, and denote by $\Pi_\PP\colon \PP M\to M$, the fiber bundle whose fiber $\PP_x$ is
 the projective space of the tangent space $T_xM$: in other word, a point $L_x\in \PP_x$ is a $1$-dimensional vector subspace of $T_xM$.
  \item We call \emph{the tautological bundle of $\PP M$}, and we  denote by $\Pi_\cT\colon \cT M\to \PP M$, the $1$-dimensional vector bundle over $\PP M$
  whose fiber $\cT_{L}$, $L\in \PP M$,  is the  the line $L$ itself.
  \item We call \emph{normal bundle of $\PP M $} and we denote by $\Pi_\cN\colon \cN M\to \PP M$, the $d-1$-dimensional vector bundle over $\PP M$ whose fiber $\cN_{L}$ over
  $L\in \PP_x$
  is the quotient space $T_x M/L$.

  If we endow $M$ with riemannian metric, then $\cN_L$  is identified with the orthogonal hyperplane of $L$ in $T_xM$.
 \end{itemize}
 \end{defi}

 Let $X$ be a $C^r$ vector field on a compact manifold  $M$, and $\phi^t$ its flow. The natural actions of the derivative of $\phi^t$ on $\PP M$ and $\cN M$ define
 $C^{r-1}$ flows on these manifolds.  More precisely, for any $t\in \RR$,

 \begin{itemize}
  \item We denote by $\phi_{\PP}^t\colon\PP M\to \PP M$ the $C^{r-1}$ diffeomorphism defined by $$\phi_{\PP}^t(L_x)= D\phi^t(L_x)\in \PP_{\phi^t(x)}.$$

  \item We denote by $\psi_{\cN}^t\colon\cN M\to \cN M$ the $C^{r-1}$ diffeomorphism  whose restriction to a fiber $\cN_L$, $L\in \PP_x$,
  is the linear automorphisms onto
  $\cN_{\phi^t_{\PP}(L)}$ defined as follows: $D\phi^t(x)$ is a linear automorphism from $T_xM$ to $T_{\phi^t(x)}M$, which maps the line $\cT_L\subset T_xM$
  onto the line $\cT_{phi^t_{\PP}(L)}$.  Therefore it passe to the quotient in the announced linear automorphism.
  $$\begin{array}[c]{ccc}
T_xM &\stackrel{D\phi^t}{\longrightarrow}&T_{\phi^t(x)}M\\
\downarrow&&\downarrow\\
\cN_L&\stackrel{\psi^t_{\cN}}{\longrightarrow}&\cN_{\phi^t_{\PP}(L)}
\end{array}$$

 \end{itemize}

 Note that $\phi^t_\PP$, $t\in \RR$ defines a flow on $\PP_M$  which is a co-cycle over $\phi^t$ whose action on the fibers is by projective maps.

 The one-parameter family   $\psi^t_\cN$ defines a flow on $\cN M$,  which is a linear co-cycle over $\phi^t_\PP$.
 We call  $\psi^t_\cN$ the \emph{extended linear Poncar\'e flow}.
 We can summarize  by the following diagrams: %\marginpar{hacer los}

 % \begin{multicols}{2}

$$\begin{array}[c]{ccc}
\cN M&\stackrel{\psi^t_{\cN}}{\longrightarrow}&\cN M\\
\downarrow&&\downarrow\\
\PP M &\stackrel{\phi_{\PP}^t}{\longrightarrow}&\PP M\\
\downarrow&&\downarrow\\
M&\stackrel{\phi^t}{\longrightarrow}&M
\end{array}$$%\columnbreak

%$$\begin{array}[c]{ccc}
%\cT M&\stackrel{\phi^t_{\cT}}{\longrightarrow}&\cT M\\
%\downarrow&&\downarrow\\
%\PP M &\stackrel{\phi_{\PP}^t}{\longrightarrow}&\PP M\\
%\downarrow&&\downarrow\\
%M&\stackrel{\phi^t}{\longrightarrow}&M
%\end{array}$$
%\end{multicols}

\begin{rema} The extended linear Poincar\'e flow is really an extension of the linear Poincar\'e flow defined in the previous section; more precisely:

 Let $S_X\colon M\setminus Sing(X)\to \PP M$ be the section of the projective bundle defined as $S_X(x)$ is the line $\langle X(x)\rangle\in \PP_x$ generated by $X(x)$.
 Then $N_X(x)= \cN_{S_X(x)}$ and the linear automorphisms $\psi^t\colon N_X(x)\to N_X(\phi^t(x))$ and $\psi^t_\cN\colon \cN_{S_X(x)}\to \cN_{S_X(\phi^t(x))}$
\end{rema}

%%%%%%%%%%%%%%%%%%%%%%%%%%%%%%%%%%%%%%%%%%%%%%%%%%%%%%%%%%%%%%%%%%%%%%%%
\subsection{Some classic theorems.}
%%%%%%%%%%%%%%%%%%%%%%%%%%%%%%%%%%%%%%%%%%%%%%%%%%%%%%%%%%%%%%%%%%%%%%%%%

We now present some results that allow us a better control of  the size of the invariant manifolds near singularities. We need for this the definition of $(\eta,T,E)^*$ contracting orbit arcs.

\begin{defi}
Given $\flow$ a flow induced by $X\in \vf$, $\La$ a compact invariant set of $\flow$, and $E\subset \norm{\La-\sing}$ an invariant bundle of the linear Poincar\'e flow  $\pflow$.
For $\eta >0$ and $T>0,$ $x\in \La-\sing$ is called \emph{$(\eta,T,E)^*$ contracting} if for any $n\in \NN$,
\[ \prod_{i = 0}^{n-1} \left\|\spflow[T|E(\phi_{iT}(x))]\right\|\leq e^{-n\eta}.\]

Similarly $x\in \La-\sing$ is called $(\eta,T,F)^*$  expanding if it is $(\eta,T,F)^*$ contracting for $-X$.
\end{defi}

To find the $(\eta,T,E)^*$ contracting orbit arcs, one needs the classical result due to V.Pliss:
\begin{lemm}\label{Pliss}\cite{P2}(Pliss lemma)
 Given a number $A$. Let $\{a_1,\cdots,a_n\}$ be a sequence of numbers  which are bounded from above by $A$.
 Assume that there exists a number $\xi<A$ such that $\sum_{i=1}^n a_i\geq n\cdot\xi$,
 then for any $\xi^{\prime}<\xi$, there exist  $l$ integers $1\leq t_1<\cdots< t_l\leq n$ such that
 $$\frac{1}{t_j-k}\sum_{i=k+1}^{t_j}a_i\geq \xi^{\prime}, \textrm{for any $j=1, \cdots,l$ and any integer $k=0,\cdots t_j-1$.}$$
 Moreover, one has the estimate $\frac l n\geq \frac{\xi-\xi^{\prime}}{A-\xi^{\prime}}.$
 \end{lemm}

\begin{rema}\label{plisspoint}
From \ref{2.7} and \ref{Pliss} we have that if the periodic orbits of a flow are volume contracting in some bundle $E$ in the period,  we can always find $(\eta,T,E)^*$ volume contracting  points.
\end{rema}
%%%%%%%%%%%%%%%%%%%%%%%%%%%%%%%%%%%%%%%%%%%%%%%%%%%%%%%%%%%%%%%%%%%%%%%%%%%%%%%%%%%%%%%%%%%%%%%%%%%%%%%%%%%%%%%%%%%%%%%%%%%%%%%%%%%%%%%%%%%%%%%%%%%%%%%%%%%%%%%%%%%%%%%%%%%%%%%%%%

\begin{lemm}[Connecting lemma]\label{l.contecting} \cite{BC}.
Given $\flow$ induced by a vector field $X\in \mathcal{X}^1(M)$ such that all periodic orbits of $X$ are hyperbolic. For any $C^1$ neighborhood $\cU$ of $X$ and $x,y\in M$ if $y$ is chain attainable from $x$, then there exists $Y\in U$ and $t>0$ such that $\flow^Y(x)= y$. Moreover the following holds: For any $k\geq 1$, let $\{x_{i,k},t_{i,k}\}_{i=0}^{n_k}$ be an $(1/k,T)$-pseudo orbit from $x$ to $y$ and denote by
$$\Delta_k = \bigcup_{i=0}^{n_k-1} \phi_{[0,t_{i,k}]}(x_{i,k}).$$
Let $\Delta$ be the upper Hausdorff limit of $\Delta_k$. Then for any neighborhood $U$ of  $\Delta$, there exists $Y\in U$ with $Y = X$ on $M\setminus U$ and $t>0$ such that $\flow^Y(x) = y$.
\end{lemm}

For a generic vector field $X\in \mathcal{X}^1(M)$ we  have:

\begin{theo}\label{ConConLem} \cite{C}
There exists a $G_{approx}\subset \mathcal{X}^1(M)$ a generic set such that for every $X\in G_{approx}$ and for every $C$ a chain recurrence class there exists a sequence of periodic orbits $\gamma_n$ which converges to $C$ in the Hausdorff topology.
\end{theo}
The following theorem by Ma\~{n}e was first introduced in \cite{Ma} and it was used in \cite{Ma2} to prove the stability conjecture. The idea behind this theorem is that the lack of hyperbolicity in a set can be detected by the clack of hyperbolicity in a periodic orbit of a $C^1$  close system.
\begin{defi}
Let $f$ be a diffeomorphism of a compact manifold $M$ with a Riemmanian metric $d$. A point $x$ is \emph{well closable} if for every $\epsilon$ there are diffeomorphisms $g$, that are $\epsilon-C^1$ close to $f$ and  periodic points $y$ for $g $ with period $T_y$, such that $$d(f^i(x),g^i(y))<\epsilon
\text{ for all }\,0\leq i\leq T_y\,.$$
We note the set of well closable points of $f$ as $\cW(f)$
\end{defi}
\begin{theo*}[\emph{Ergodic closing lemma}]
Let $f$ be a diffeomorphism and $\mu$ an $f-$invariant probability measure, then almost every point is well closable. That is $$\mu\cW(f)=1\,.$$
\end{theo*}

The version of this theorem for flows is almost the same with one exception: the well closable points might be closed by a singularity.

\begin{defi}
Let $X$ be a vector field  of a compact manifold $M$ with a Riemannian metric $d$, and $\phi$ its associated flow. A point $x$ is \emph{well closable} if for every $\epsilon$ there are vector fields $Y$, that are $\epsilon-C^1$ close to $X$ and critical elements (closed orbits) $y$ for $Y $ with period $T_y$, such that $$d(\phi^X_t(x),\phi^Y_t(y))<\epsilon
\text{ for all }\,0\leq t\leq T_y\,.$$
We note the set of well closable points of $X$ as $\cW(X)$
\end{defi}
\begin{theo*}[\emph{Ergodic closing lemma for flows}]\label{l.ergodic}
Let $X$ be a vector field  of a compact manifold $M$ with a Riemannian metric $d$. For every $T>0$ and $\mu$ a $\phi_T-$invariant probability measure,  almost every point is well closable. That is $$\mu\cW(f)=1\,.$$
\end{theo*}

\section{Singular volume partial hyperbolicity}

\subsection{Strong stable, strong unstable and center spaces associated to a hyperbolic singularity.}
%%%%%%%%%%%%%%%%%%%%%%%%%%%%%%%%%%%%%%%%%%%%%%%%%%%%%%%%%%%%%%%%%%%%%%%%%%%%%%%%%%%%%%%%%%%%%%%%%%%%%%%%%%%%%%%%%%%%%%%%%%%%%%
Let $X$ be a vector field and $\sigma\in Sing(X)$ be a hyperbolic singular point of $X$.
Let
 $\lambda^s_k\dots\lambda^s_2<\lambda^s_1<0<\lambda^u_1<\lambda^u_2\dots \lambda^u_l$ be the Lyapunov exponents of $\phi_t$ at $\sigma$ and let
 $E^ s_k\oplus_{_<}\cdots E^s_2\oplus_{_<}E^s_1\oplus_{_<}E^u_1\oplus_{_<}E^u_2\oplus_{_<}\cdots \oplus_{_<}E^s_l$ be the corresponding (finest) dominated
 splitting over $\sigma$.

 A subspace $F$ of $T_\sigma M$ is called a \emph{center subspace} if it is of one of the possible form below:
 \begin{itemize}
  \item Either $F=E^ s_i\oplus_{_<}\cdots E^s_2\oplus_{_<}E^s_1$
  \item Or $F=E^u_1\oplus_{_<}E^u_2\oplus_{_<}\cdots \oplus_{_<}E^s_j$
  \item Or else $F=E^ s_k\oplus_{_<}\cdots E^s_2\oplus_{_<}E^s_1\oplus_{_<}E^u_1\oplus_{_<}E^u_2\oplus_{_<}\cdots \oplus_{_<}E^s_l$
 \end{itemize}

 A subspace  of $T_\sigma M$ is called a \emph{strong stable space}, and we denote it  $E^{ss}_{i}(\sigma)$,  if there in $i\in\{1,\dots, k\}$ such that:
 $$E^{ss}_{i}(\sigma)=E^ s_k\oplus_{_<}\cdots E^s_{j+1}\oplus_{_<}E^s_i$$

 A classical result from hyperbolic dynamics asserts that for any $i$ there is a unique injectively immersed manifold $W^{ss}_i(\sigma)$, called a
 \emph{strong stable manifold}
 tangent at $E^{ss}_i(\sigma)$ and invariant by the flow of $X$.

 We define analogously the \emph{strong unstable spaces} $E^{uu}_j(\sigma)$ and the \emph{strong unstable manifolds} $W^{uu}_j(\sigma)$ for $j=1,\dots ,l$.

%%%%%%%%%%%%%%
 %%%%%%%%%%%%%%%%%%%%%%%%%%%%%%%%%%%%%%%%%%%%%%%%%%%%%%%%%%%%%%%%%%%%%%%%%%%%%%%%%%%%%%%%%%%%%%%%%%%%%%%%%%%%%%

\subsection{The lifted maximal invariant set and the singular points}\label{ss.centerspace}
%%%%%%%%%%%%%%%%%%%%%%%%%%%%%%%%%%%%%%%%%%%%%%%%%%%%%%%%%%%%%%%%%%%%%
Let $\La$ be a maximal invariant set for a flow $X$.
We define de \emph{lifted maximal invariant set}, that is:
$$\La_{\PP,U}(X)=\overline{\set{<X(x)>\in\PP M \text{ such that } x\in\La}}\,.$$
The lifted maximal invariant set does not vary upper  semi-continuously with $X$.

 Let $U$ be a compact region, and $X$ a vector field. Let $\sigma\in Sing(X)\cap U$ be a hyperbolic singularity of $X$ contained in  $U$.
 We are interested on $\La_{\PP,U}(X)\cap \PP_\sigma$. As in \cite{BdL}  the aim of this section is to add  to
 the lifted maximal invariant set $\La_{\PP,U}$, some set over the singular
 points in order to recover some upper semi-continuity properties.

\subsubsection{Upper semi continuity of the lift}

We can consider now the smallest set that varies upper semi continuously with the vector field that was introduced in \cite{GLW} .  The set is defined as follows:
\begin{defi}\label{d.chinese} Let $U$ be a compact region and $X$ a $C^1$  vector field. Let $\cU$ be a neighborhood of $X$
 Then we define 
 $$\widetilde{\La}=\overline{\set{<Y(x)>\in\PP M \text{ such that } x\in U\cap Per(Y) \,\,\text{ and } \,\,Y\in \cU}}\,.$$

 \end{defi}
When the hypothesis of our problem gives us information about an open set of vector fields, for instance when we are talking about a robustly transitive sets, this lifted set of directions prooves very easy to work with. However, when we intend to define a dominated splitting structure or a partial hyperbolic structure over it, then  we want to do it over a set that one can detect by looking at only one vector field. Therefore in \cite{BdL} a bigger set is introduced, that does not relay on information of the perturbations of the flow and varies upper semi continuously with the vector field.
\subsection{The extended maximal invariant set}

We define the \emph{escaping stable space} $E^{ss}_{\sigma,U}$ as the biggest strong stable space $E^{ss}_j(\sigma)$ such that the corresponding strong stable manifold
$W^{ss}_j(\sigma)$  is \emph{escaping}, that is:  $$\Lambda_{X,U}\cap W^{ss}_j(\sigma)=\{\sigma\}.$$

We define the \emph{escaping  unstable space} analogously.

We define the \emph{central space $E^c_{\sigma,U}$ of $\sigma$ in $U$}  the center space such that
$$T_\sigma M=E^{ss}_{\sigma,U}\oplus E^ c_{\sigma,U}\oplus E^ {uu}_{\sigma,U}$$

 We denote by
 $\PP^i_{\sigma,U}$ the projective space of $E^ i(\sigma,U)$ where $i\in\set{ss,uu,c}$.
The proofs of all lemmas proppositions and theorems  in this subsection can be found in \cite{BdL}

\begin{lemm}
The central space  $E^c_{\sigma,U}$   is the smallest center space containing $\La_{\PP,U}\cap \PP_\sigma$.
\end{lemm}
\begin{lemm}\label{l.lower} Let $U$ be a compact region. Let $\sigma$ be a hyperbolic singular point in $U$,  that has a continuation $\sigma_Y$ for vector fields
$Y$ in a $C^1$-neighborhood of $X$. Then both  escaping strong stable and unstable spaces $E^{ss}_{\sigma_Y,U}$ and $E^{uu}_{\sigma_Y,U}$ depend lower semi-continuously on
$Y$.

As a consequence the central space $E^c_{\sigma_Y,U}$  of $\sigma_Y$ in $U$ for $Y$
depends upper semi-continuously on $Y$,  and the same happens for its projective space
$\PP^s_{\sigma_Y,U}$.
\end{lemm}

\begin{defi}Let $U$ be a compact region and $X$  a vector field whose singular points are hyperbolic.
 Then the set
 $$B(X,U)=\La_{\PP,U}\cup\bigcup_{\sigma\in Sing(X)\cap U} \PP^c_{\sigma,U} \subset \PP M$$
 is called the \emph{extended maximal invariant set of $X$ in $U$}
 \end{defi}

 \begin{prop}\label{p.extended} Let $U$ be a compact region and $X$  a vector field whose singular points are hyperbolic.
Then the extended maximal invariant set $B(X,U)$ of $X$ in $U$ is a compact subset  of $ \PP M$, invariant under the flow $\phi_\PP^t$.
Furthermore, the map $X\mapsto B(X,U)$ depends upper semi-continuously on $X$.
 \end{prop}

 \begin{prop}\label{p.extended} Let $U$ be a compact region and $X$  a vector field whose singular points are hyperbolic.
Then the extended maximal invariant set $B(X,U)$ of $X$ in $U$ is a compact subset  of $ \PP M$, invariant under the flow $\phi_\PP^t$.
Furthermore, the map $X\mapsto B(X,U)$ depends upper semi-continuously on $X$.
 \end{prop}

\subsection{Extending chain recurrence classes}
Conley theory asserts that  any chain recurrent  $C$ class admits a basis of neighborhood which are nested filtrating neighborhood $U_{n+1}\subset U_n$, $C=\bigcap U_n$
  We define

$$\widetilde{\La(C)}=\bigcap_n\widetilde{\La(X,U_n)}  \mbox{ and }B(C)=\bigcap_n B(X,U_n).$$
These two sets are independent of the choice of the sequence $U_n$ and $\widetilde{\La(C)}\subset B(C)$.

\begin{defi}
We say that a chain recurrence class $C$ has a given singular hyperbolic structure if $\widetilde{\La(C)}$ carries this singular hyperbolic structure.
\end{defi}

If $\sigma\in Sing(X)$ is an hyperbolic singular point we define $E^c_\sigma=\bigcap_n E^c_{\sigma,U_n}$ and we call it the center space of $\sigma$.
We denote $\PP^c_\sigma =\PP E^c_\sigma$, its projective space.

\begin{rema}Consider the open and dense set of vector fields whose singular points are all hyperbolic.  In this open set the singularities depend continuously on the field.
Then for every singular point $\sigma$, the projective center space $\PP^c_\sigma$ varies upper semi continuously, and in particular the dimension  $dim E^c_\sigma$ varies upper semi-continuously.
As it is a non-negative integer, it is locally minimal and locally constant on a  open and dense subset.

We will say that such a singular point has \emph{locally minimal center space}.
\end{rema}

%%%%%%%%%%%%%%%%%%%%%%%%%%%%%%%%%%%%%%%%%%%%%%%%%%%%%%%%%%%%%%%%%%%%%%%%%%%%%%%%%%%%%%%%%%%%%%%%%%%%%%%%%%%%%%%%%%%%%%%%%%%%%%%%%%%%%%%%%%%%%%%%%%%%%%%%%%%%%%%%%%%%%%%%%%%%%%%%%%%%%%%%%%%%%%%%%%%%%%%%
\subsection{Reparametrizations}
Let $\cA=\{A^t(x)\}$ and $\cB=\{B^t(x)\}$ be two linear cocycles on the same linear bundle $\pi\colon \cE\to \La$ and over the  same flow $\phi^t$ on a compact invariant set $\La$ of a manifold $M$. We say that
$\cB$ is  a \emph{reparametrization} of $\cA$ if there is a continuous map $h=\{h^t\}\colon \La\times \RR\to (0,+\infty)$ so that for every $x\in\La$ and $t\in\RR$ one has
$$B^t(x)=h^t(x) A^t(x).$$
The reparametrizing map $h^t$  satisfies the cocycle relation
$$h^{r+s}(x)=h^r(x)h^s(\phi^r(x)),$$
and is called a \emph{cocycle}.

One easily check the following lemma:
\begin{lemm}
Let $\cA$ be a linear cocycle and $\cB$ be a reparametrization of $\cA$.  Then  any dominated splitting for $\cA$ is a dominated splitting for $\cB$.
\end{lemm}

\begin{rema} \begin{itemize}\item If $h^t$ is a cocycle, then for any $\alpha\in\RR$  the power $(h^t)^\alpha: x\mapsto (h^t(x))^\alpha$ is a cocycle.
\item If $f^t$ and $g^t$ are cocycles then $h^t=f^t\cdot g^t$ is a cocycle.
\end{itemize}
\end{rema}

A cocycle $h^t$ is called a \emph{coboundary} if there is a continuous function $h\colon \La\to (0,+\infty)$ so that

$$h^t(x)= \frac{h(\phi^t(x))}{h(x)}.$$

A coboundary cocycle in uniformly bounded.
Two cocycles $g^t$, $h^t$ are called \emph{cohomological} if $\frac{g^t}{h^t}$ is a coboundary.

\begin{rema}\label{r.cohomological}
The cohomological relation is an equivalence relation among the cocycle and is compatible with the product: if $g^t_1$ and $g^t_2$ are cohomological
and $h^t_1$ and $h^t_2$ are cohomological then $g^t_1h^t_1$ and $g^t_2 h^t_2$ are cohomological.
\end{rema}

\begin{lemm}
if $g$ and $h$ are cohomological then $g\cdot \cA$ is hyperbolic if and only if  $h\cdot \cA$ is hyperbolic.
\end{lemm}
%%%%%%%%%%%%%%%%%%%%%%%%%%%%%%%%%%%%%%%%%%%%%%%%%%%%%%%%%%%%%%%%%%%%%%%%%%%%%%%%%%%%%%%%%%%%%%%%%%%%%%%%%%%%%%%%%%%%%%%%%%%%%%%%%%%%%%%%%
\subsubsection{Reparametrizing cocycle associated to a singular point}\label{ss.reparametrization}
%%%%%%%%%%%%%%%%%%%%%%%%%%%%%%%%%%%%%%%%%%%%%%%%%%%%%%%%%%%%%%%%%%

Let $X$ be a $C^1$ vector field, $\phi^t$ its flows,  and $\sigma$ be a hyperbolic singularity of $X$.
We denote by $\Lambda_X\subset\PP M$ the union
$$\Lambda_X = \overline{\{\RR X(x), x\notin Sing(X)\}}\cup\bigcup_{x\in Sing(X)} \PP T_x M.$$

It can be shone easily that this set  is upper semi-continuous, as in the case of $B(X,U)$ (see \ref{p.extended} )

\begin{lemm} $\La_X$ is a compact subset of $\PP M$ invariant under the flow $\phi_\PP^t$, and the map $X\mapsto \La_X$ is upper semi-continuous.
Finally, if the singularities of $X$ are hyperbolic then, for any compact regions one has $B(X,U)\subset \La_X$.
\end{lemm}

Let $U_\sigma$ be a compact neighborhood of $\sigma$ on which
$\sigma$ is the maximal invariant.

Let $V_\sigma$ be a compact neighborhood of $Sing(X)\setminus\{\sigma\}$ so that $V_\sigma\cap U_\sigma=\emptyset$.
We fix a ($C^1$) Riemmann metric $\|.\|$ on $M$ so that
$$\|X(x)\|=1 \mbox{ for all } x\in M\setminus (U_\sigma\cup V_\sigma).$$

Consider the map $h\colon \Lambda_X\times \RR\to (0,+\infty)$, $h(L,t)=h^t(L)$, defined as follows:
\begin{itemize}
 \item if $L\in \PP T_xM$ with   $x\notin U_\sigma$ and $\phi^t(x)\notin U_\sigma$, then $h^t(L)=1$;
 \item if $L\in \PP T_xM$ with $x\in U_\sigma$ and $\phi^t(x)\notin U_\sigma$ then $L= \RR X(x)$  and
 $h^t(L)= \frac 1{\|X(x)\|}$;
 \item if $L\in \PP T_xM$ with $x\notin U_\sigma$ and $\phi^t(x)\in U_\sigma$ then $L= \RR X(x)$  and
 $h^t(L)= \|X(\phi^t(x))\|$;
 \item  if $L\in \PP T_xM$ with $x\in U_\sigma$ and $\phi^t(x)\in U_\sigma$ but $x\neq \sigma$ then $L= \RR X(x)$  and
 $h^t(L)= \frac {\|X(\phi^t(x)\|}{\|X(x)\|}$;
 \item if $L\in \PP T_\sigma M$ then $h^t(L)= \frac {\|\phi_\PP^t(u)\|}{\|u\|}$ where $u$ is a vector in $L$.
\end{itemize}
Note that the case in which $x$ is not the singularity and  $x\in U_\sigma$, then $h$ can be written as in the last item, taking $u=X(x)$.

As it was proven in \cite{BdL} the map we just defined is  a (continuous) cocycle on $\La_X$ and  it's cohomology class,  is independent from the choice of the metric $\|.\|$ and of the neighborhoods.

Also from \cite{BdL} we have :
\begin{lemm}\label{l.representativecocycle} Consider a vector field $X$ and a hyperbolic singularity $\sigma$ of $X$. Then there is a $C^1$-neighborhood $\cU$ of $X$ so that $\sigma$ has a well defined hyperbolic
continuation $\sigma_Y$ for $Y$ in $\cU$ and for any $Y\in\cU$ there is a map $h_Y\colon \La_Y\times \RR\to (0,+\infty)$ so that
\begin{itemize}
 \item for any $Y$, $h_Y$ is a cocycle belonging to the cohomology class $[h(Y,\sigma_Y)]$
 \item $h_Y$ depends continuously on $Y$:  if $Y_n\in \cU$ converge to $Z\in \cU$ for the $C^1$-topology and if $L_n\in \La_{Y_n}$ converge to $L\in \La_Z$ then
 $h^t_{Y_n}(L_n)$ tends to $h^t_Z(L)$ for every $t\in \RR$; furthermore this convergence is uniform in $t\in[-1,1]$.
\end{itemize}
\end{lemm}

%%%%%%%%%%%%%%%%%%%%%%%%%%%%%%%%%%%%%%%%%%%%%%%%%%%%%%%%%%%%%%%%%%%%%%%%%%%%%%%%%%%%%%%%%%%%%%%%%%%%%%%%%%%%%%%%%%%%%%%%%
%%%%%%%%%%%%%%%%%%%%%%%%%%%%%%%%%%%%%%%%%%%%%%%%%%%%%%%%%%%%%%%%%%%%%%%%%%%%%%%%%%%%%%%%%%%%%%%%%%%%%%%%%%%%%%%%%%%%%%%%%%%%%%%%%%%%%%%%%%%%%%%%%%%%%%%%%%%%%%%%%%%%%%%%%%%%

%%%%%%%%%%%%%%%%%%%%%%%%%%%%%%%%%%%%%%%%%%%%%%%%%%%%%%%%%%
\section{Definition of singular volume partial hyperbolicity}
%%%%%%%%%%%%%%%%%%%%%%%%%%%%%%%%%%%%%%%%%%%%%%%%%%%%%%%%%%%%

Let $X$ be a vector field with a dominated splitting of the linear Poincar\'e flow  $$\cN_L=\cN^s\oplus\dots\oplus \cN^u$$ where $\cN^s_L$ and $\cN^u_L$ denote the extremal bundles, over a chainrecurrence class $C$. Supose as well that   the set $S$ of the singularities in $C$ contains only hyperbolic singularities.
We subdivide $S$ as follows 

\begin{itemize}\item the set $S_{Ec}$ of singular points whose escaping stable space has  dimension smaller than $\cN_L^s$,
\item the set $S_{E}$ of singular points whose escaping stable space has  dimension bigger or equal  than $\cN_L^s$,
\item the set $S_{Fc}$ of singular points whose escaping unstable space has  dimension smaller  than $\cN_L^u$,
\item the set $S_{F}$ of singular points whose escaping unstable space has  dimension bigger or equal  than $\cN_L^u$.
\end{itemize}

\begin{defi}\label{defisvh}Let $X$ be a $C^1$-vector field on a compact manifold and let $C$ bea chain recurrence class. We denote $S=Sing(X)\cap C$.
One says that $X$ is \emph{singular volume partial hyperbolic} on $C$ if
\begin{enumerate}
 \item The restriction of the extended linear Poincar\'e flow $\{\psi_\cN^t\}$ to the  extended maximal invariant set $B(C)$ admits
 a finest dominated splitting $$\cN_L=\cN^s\oplus\dots\oplus \cN^u$$ where $\cN^s_L$ and $\cN^u_L$ denote the extremal bundles,
and $n_s$ and $n_u$ their respective dimensions.

 \item There is a subset $S_{Ec}\subset S$ so that the reparametrized cocycle $h_{Ec}^t\psi_\cN^t$ contracts volume in restriction to the
 bundles $\cN^s_L$ over  $B(C)$   where $h_{Ec}$ denotes
 $$h_{Ec}=\Pi_{\sigma\in S_{Ec}} h_\sigma^{n_s}.$$
%\item  In the complement of $S_{Ec}$ in $S$ that we note $S_{E}\subset S$, $\psi_\cN^t$ contracts volume in restriction to the
 %bundles $N^s_L$ over   $B(C)$ .
 \item There is a subset $S_{Fc}\subset S$ so that the reparametrized cocycle $h_{Fc}^t\psi_\cN^t$  expands volume in restriction to the
 bundle $\cN^u_L$ over $B(C)$  where $h_{Fc}$ denotes
 $$h_{Fc}=\Pi_{\sigma\in S_{Fc}} h_\sigma^{n_u}.$$
%\item  In the complement of $S_{Fc}$ in $S$ that we note $S_{F}\subset S$, $\psi_\cN^t$ expands volume in restriction to the
% bundles $N^u_L$ over $B(C)$ .
\end{enumerate}
\end{defi}

\begin{rema}
  If $C$ is a chain recurrence class which is singular volume  hyperbolic then $X$ is singular volume  hyperbolic  on a small filtrating neighborhood of $C$.
\end{rema}

\begin{defi}\label{defieq}Let $X$ be a $C^1$-vector field on a compact manifold and let $\La$ be a maximal invariant set in $U$. We denote $S=Sing(X)\cap U$.
One says that $X$ is \emph{singular volume partial hyperbolic over  $\widetilde{\La}(X,U)$  } if
\begin{enumerate}
 \item The restriction of the extended linear Poincar\'e flow $\{\psi_\cN^t\}$ to the  extended maximal invariant set $\widetilde{\La}(X,U)$admits
 a finest dominated splitting $\cN_L=\cN^s\oplus\dots\oplus \cN^u$ where $\cN^s_L$ and $\cN^u_L$ denote the extremal bundles.
 \item There is a subset $S_{Ec}\subset S$ so that the reparametrized cocycle $h_{Ec}^t\psi_\cN^t$ contracts volume in restriction to the
 bundles $\cN^s_L$ over  $\widetilde{\La}(X,U)$   where $h_{Ec}$ denotes
 $$h_{Ec}=\Pi_{\sigma\in S_{Ec}} h_\sigma^{n_s}.$$
%\item  In the complement of $S_{Ec}$ in $S$ that we note $S_{E}\subset S$, $\psi_\cN^t$ contracts volume in restriction to the
 %bundles $N^s_L$ over    $\widetilde{\La}(X,U)$  .
 \item There is a subset $S_{Fc}\subset S$ so that the reparametrized cocycle $h_{Fc}^t\psi_\cN^t$  expands volume in restriction to the
 bundle $\cN^u_L$ over $\widetilde{\La}(X,U)$  where $h_{Fc}$ denotes
 $$h_{Fc}=\Pi_{\sigma\in S_{Fc}} h_\sigma^{n_u}.$$
%\item  In the complement of $S_{Fc}$ in $S$ that we note $S_{F}\subset S$, $\psi_\cN^t$ expands volume in restriction to the
% bundles $N^u_L$ over $\widetilde{\La}(X,U)$  .
\end{enumerate}
\end{defi}

\begin{rema}\label{remaeq}
If $C$ is a robustly chain transitive calss it follows that there is a neighborhood $U$ of $C$ such that   $\widetilde{\La}(X,U)=\widetilde{\La(C)}$ .

\end{rema}

In \cite{BdL} it is shonw that the hyperbolic structures of  $\widetilde{\La(C)}$ extend to $B(C)$. This is fundamental to our purpose, since in our scenario, dealing with $\widetilde{\La(C)}$  is much more convinient. 
\begin{theo}[BdL]\label{t.nioki} Let $X$ be a vector field on a closed manifold, whose singular points are all hyperbolic and with locally minimal center spaces. Then for every
$\sigma\in Sing X$, a singular volume hyperbolic structure  on $\widetilde{\La(C)}$ extends to $B(C)$.
\end{theo}

\begin{rema}\label{r.defieq}
From remark \ref{remaeq} and theorem  \ref{t.nioki} we get that if $C$ is robustly chain transitive, then for $X$   under the hypothesis of the theorem
 the definitions \ref{defisvh} and \ref{defieq} are in fact equivalent.
\end{rema}

\section{Analysis of the hyperbolicity in the closed curves }
In this section we prove that an open and dense subset of the vector fields having a robustly chain transitive chain recurrence class $C$ has a dominated splitting over $B(C)$. A dominated splitting over $\widetilde{\La(C)}$, from\cite{BdL}, and since for an isolated neighborhood of $C$, $U$, we have that  $\widetilde{\La(C)}=\widetilde{\La}(X,U)$. Consequently we just need to find a dominated splitting over $\widetilde{\La}(X,U)$  which is a rather straight forwards consequence of already existing results.

%%%%%%%%%%%%%%%%%%%%%%%%%%%%%%%%%%%%%%%%%%%%%%%%%%%%%%%%%%%%%%%%%%%%%%%%%%%%%%%%%%%%%%%%%%%%%%%%%%%%%%%%%%%%%%%%%%%%%%%%%%%%%%%%%%%%%%%%%%%%%%%%%%%%%%%%%%%%%%%%%%%%%%%%%%%%%%%%%%%%%%%%%%%%%%%%%%%%%%%%%%%%%%%%%%%%%%%%%%%%%%%%%%%%%%%%%%%%%%%%%%%%%%
\subsection{The set of periodic orbits}\label{s.per}

In a series of papers such as \cite{DPU}  \cite{BDP}  \cite{BB},and  \cite{BGV}  it is shown  that if a set of periodic orbits, that has periodic orbits of arbitrarily long periods, does not have a dominated splitting, then there is a perturbation of the flow having a an infinite number of sinks and sources.
We state two of this result below:

 \begin{lemm}[ \cite{BGV} Theorem 2.2]\label{l.perdom}
Let $\cA$ be a bounded linear cocycle over $\pi:E\to\sum$  where $\sum$ is a set of periodic orbits, containing periodic orbits of arbitrarily long periods.
Then if $\cA$  does not admit any dominated splitting, then there exist a perturbation $\cB$ of $\cA$ such that $E$ is contracted or expanded by $\cB$ along the orbit.
 \end{lemm}

 \begin{lemm}[ \cite{BDP} lemma 6.1]\label{l.percontr}
Let $\cA$ be a bounded linear cocycle over $\pi:E\to\sum$  where $\sum$ is a set of periodic orbits, containing periodic orbits of arbitrarily long periods. Let $T(x)$ denote the period of $x\in\sum$
Suppose that $\cA$
\begin{itemize}
  \item $\cA$ admits a dominated splitting $E=F_1\oplus_{\prec}F_2\,,$
  \item $\cA$ does not  admit a dominated splitting of $F_1$.
  \item  There exist a point $p$  such that $\det (A^{T(p)}\mid_{F_1}(p))>1$
\end{itemize}
then there exist a perturbation $\cB$ of $\cA$ and q in $\sum$ such that all eigen values of $A^{T(q)}\mid_{F_1}(q)$  are positive.
 \end{lemm}

 \begin{rema}\label{r.percontr}
 Both of these theorems hold if we consider the linear Poincar\'{e} flow as the cocycle, and the time $t$ of the flow as the diffeomorphism.
Therefore the set of periodic orbits in a robustly chain transitive set $C$ has dominated splitting and the finest possible of  this dominated splittings is such that the extremal bundles contract and expand volume.

\end{rema}

 \begin{lemm}\label{l.chinesdom}
There is an open and dense subset of the vector fields  $X\in\cX^1M $ with a robustly chain transitive set $C$  t in a filtrating neighborhood $U\subset M $ such that set $\widetilde{\La}(X,U)$
admits a finest dominated splitting of the normal bundle, for the extended linear Poincar\'e flow.
 \end{lemm}

\proof
Take $X$ a vector field in the hypothesis of  theorem \ref{ConConLem} we can find a sequence of vector fields $Y_n$ converging to $X$ such that the set $C$ is the hausdorff limit of the periodic orbits of the $Y_n$. Since $C$ is robustly chain transitive this periodic orbits must be related, so for $Y_n$,  $C$ is the closure of $ U\cap Per(Y_n)$.
From lemma \ref{l.perdom} the set $\set{<Y_n(x)>\in\PP M \text{ such that } x\in \La }$ has a uniform finest dominated splitting for the extended linear Poincar\'{e} flow (note that the projection to $M$ here is one to one).   Therefore,  the extended linear Poincar\'{e} flow has dominated splitting over the set $$\overline{\set{<Y_n(x)>\in\PP M \text{ such that } x\in C}}\,,$$  since a uniform dominated splitting  extends to the closure.

Using again \ref{ConConLem},  the sequence $Y_n$ can be taken so that the hausdorff limit of $$\overline{\set{<Y_n(x)>\in\PP M \text{ such that } x\in C }}\,,$$ is exactly
the set  $\widetilde{\La(X,U)}$, since 
the upper limit coincides with the Hausdorff limit if ti exist, therefore  we get that  $\widetilde{\La(X,U)}$ has a dominated splitting of the same dimension, for the extended linear Poincar\'{e} flow of $X$
\endproof

%%%%%%%%%%%%%%%%%%%%%%%%%%%%%%%%%%%%%%%%%%%%%%%%%%%%%%%%%%%%%%%%%%%%%%%%%%%%%%%%%%%%%%%%%%%%%%%%%%%%%%%%%%%%%%%%%%%%%%%%%%%%%%%%%%%%%%%%%%%%%%%%%%%%%%%%%%%%%%%%%%%%%%%%%%%%%%%%%%%%%%%%%%%%%%%%%%%%%%%%%%%%%%%%%%%%%%%%%%%%%%%%%%%%%%%%%%%%%%%%%%%%%%

\section{Analyzing the singulatiries}\label{s.sings}

The main problem we need to deal with now,  is the distortion of the contraction and expansion rates that occurs when the periodic orbits approach the singularities. For this we will use again the reparametrized linear Poincar\'{e} flow.

With this tool, the main  work in this section will be to find out which singularities need to be reparametrized. After this and following mainly \cite{BDP},  we show that all closed orbits in an open set of vector fields have the structure desired. 

Then, in the next section we follow the classical strategy. We use the ergodic closing lemma, to argue that if the robust chain transitive set did not have the desired structure, then there must be a critical element in a perturbed vector field that does not have that structure either. This contradiction, gives us our main theorem \ref{t.BDP}
%%%%%%%%%%%%%%%%%%%%%%%%%%%%%%%%%%%%%%%%%%%%%%%%%%%%%%%%%%%%%%%%%%%%%%%%%%%%%%%%%%%%%%%%%%%%%%%%%%%%%%%%%%%%%%%%%%%%%%%%%%%%%%%%%%%%%%%%%%%%%%%%%%%%%%%%%%%%%%%%%%%%%%%%%%%%%%%%%%%%%%%%%%%%%%%%%%%%%%%%%%%%%%%%%%%%%%%%%%%%%%%%%%%%%%%%%%%%%%%%%%%%%%

\subsection{Center space of the singularities and the dominated splitting  on $\widetilde{\La}(X,U)$}

%Now we want to extend the dominated splitting on the projective space of the singularity to the lifted maximal invariant set $B(X,U)$.
Let us consider a singularity $\sigma\in C$. We consider the following splitting of its tangent space:
$$E^{ss}\oplus E^c\oplus E^{uu}\,,$$ noting the stable escaping, the unstable escaping and the center spaces.
We suppose the singularities to be hyperbolic

Recall that from lemma \ref{l.chinesdom} we have that there is a finest
 dominated splitting for the extended linear Poincar\'{e} flow, over $\widetilde{\La}(X,U)$.  We note this as follows
$$\cN_L=\cN^s\oplus\cN^1\oplus\dots\oplus\cN^{k}\oplus\cN^u\,,$$
where $L$ is a direction in $\widetilde{\La}(X,U)$.

We call $\pi_L:T_xM\to \cN_L$ where $L\in\mathbb{P}_xM$ the projection over the normal space at a given direction $L$.

The next lemma from \cite{BdL} tells us that there is a relation between the splitting in the singularities into escaping and center spaces, and the dominated splitting of the class for the linear Poincar\'e flow.
\begin{lemm}[BdL]\label{l.central}
Let us consider the set of vector fields $\cV$ such that evey  vector field $X\in\cV$  has a singular chain class $C_{\sigma}$ with the following properties.  We denote $S= Sing(X)\cap C_{\sigma}$\begin{itemize}
                                                           \item  every $\sigma\in S$ that is hyperbolic,
                                                           \item the dimension of the central space of $\sigma\in S$ is locally constant,

                                                           \item the extended linear Poincar\'{e} flow over $\widetilde{\La}(C(\sigma))$ has a  dominated splitting,
$$\cN_L=\cN^E\oplus\cN^F\,,$$
where $L$ is any direction in $\widetilde{\La}(C(\sigma))$.
                                                         \end{itemize}

Let $L $ be a direction in $\widetilde{\La}\cap\PP_{\sigma}M$, %and such that $L=<u>$ where $u$ belongs to some $E^{si}$
Then there is an open and dense subset of $\cV$ noted $\cU\subset\cV$ such that for every $X\in\cU$ , $$\pi_L ( E_{\sigma}^c)\subset\cN_L^E\,,$$ or
$$\pi_L (E_{\sigma}^c)\subset\cN_L^F\,.$$
 \end{lemm}

The next corollary is a direct consequence of the previous lemma.

\begin{coro}
Let us consider the set of vector fields $\cV$ such that evey  vector field $X\in\cV$  has a singular chain class $C_{\sigma}$  such that if we note $S= Sing(X)\cap C_{\sigma}$ \begin{itemize}
                                                           \item  every $\sigma\in S$ that is hyperbolic,
                                                           \item the dimension of the central space of $\sigma\in S$ is locally constant,
\item   the extended linear Poincar\'{e} flow over $\widetilde{\La}(C(\sigma))$ has a  finest  dominated splitting $$\cN_L=\cN^s\oplus\cN^1\oplus\dots\oplus\cN^{k}\oplus\cN^u\,,$$
\end{itemize}
Let $L $ be a direction in $\widetilde{\La}(X,U)\cap\PP_{\sigma}M$, and such that $L=<u>$.
Then  there is an open and dense subset of $\cV$ noted $\cU\subset\cV$ such that for every $X\in\cU$, $$\pi_L ( E_{\sigma}^c)\cap\cN_L^s=\emptyset\,,$$ or
$$\pi_L ( E_{s1}\oplus\dots\oplus E_{sk}\oplus E_{u1})\subset\cN_L^s\,.$$
 \end{coro}

Now we can see a more precise definition of the sets in definition \ref{defisvh}.
\begin{defi}
Let $X$ be a $C^1$ vector field, such that there is an open set $U$ such the maximal invariant set in $U$ is a robustly chain transitive chain recurrence class.
Let us consider the finest dominated splitting for the set $\widetilde{\La}(X,U)$, $$\cN_L=\cN^s_L\oplus\dots\oplus \cN^U_L\,.$$
We define $S_{Ec}$ the set of singularities $$S_{Ec}=\set{\sigma\in Sing(X)\cap U \text{such that } dim(\cN^s_L)>E^{ss}_{\sigma}}\,.$$
Similarly we define the set $S_{Fc}$ the set of singularities $$S_{Fc}=\set{\sigma\in Sing(X)\cap U \text{such that } dim(\cN^u_L)>E^{uu}_{\sigma}}\,.$$

\end{defi}
\begin{rema}
This definition makes the set in definition \ref{defisvh} to de uniquely defined and disjoint
\end{rema}

\subsection{Volume contraction at the singularities.}
Recall that for a hyperbolic singularity we note 
$$T_{\sigma}M=E^{ss}\oplus E^c\oplus E^{uu}\,,$$ noting the stable escaping, the unstable escaping and the center spaces.
 we write the center space as:
$$E_{\sigma}^c= E_{s1}\oplus\dots\oplus E_{sk}\oplus E_{u1}\oplus\dots\oplus E_{ul}$$

\begin{lemm}\label{l.repcont}
Let $\sigma$ be a singularity of $C$, a robustly transitive chain recurrence class with an isolating filtrating  neighborhood $U$. Let $\Gamma =Orb(x)$ be a homoclinic orbit
associated to $\sigma$ . Assume as well that:

\begin{itemize}
\item There exists a sequence of vector fields $X_n$
converging to $X$ in the $C^1$ topology

\item There exist a sequence of  periodic orbit $\gamma_n$ of $X_n$  such
that $\gamma_n$ converges to $\Gamma$ in the Hausdorff topology.

\item There is a finest dominated splitting over  $\widetilde{\La}(X,U)$, $$\cN_L=\cN^s\oplus\cN^1\oplus\dots\oplus\cN^{k}\oplus\cN^u$$ where $L\in\widetilde{\La}(X,U)$. We note   $dim(\cN^u)=h$ and  we have that $dim(\cN^s)=n$.
%\item There is a sequence of points $q_n\in\gamma_n$ such that $X_n(q_n)\to u $ is in  $E_{u1}\oplus\dots\oplus E_{ul}$.
%\item We call $L_u=<u>$ and we have that $dim(\cN^u_{L_u})=h$
% \item There is a sequence of points $p_n\in\gamma_n$ such that $<X(p_n)>\to v $ is in  $E_{s1}\oplus\dots\oplus E_{sk}$
% \item we call $L_v=<v>$
%and  we have that $dim(\cN^s_{L_v})=n$.
\end{itemize}

Then there is a space $E\subset T_{\sigma}M$ such that $E$ contracts volume and has dimension $n+1$, and a $F\subset T_{\sigma}M$ such that $F$ expands volume and has dimension $h+1$
 \end{lemm}
 \proof
Let us recall that from remark \ref{r.percontr} the splitting $$\cN_L=\cN^s\oplus\cN^1\oplus\dots\oplus\cN^{k}\oplus\cN^u$$ where $L$ is a direction over the set of periodic orbits of $C$ is Volume partial hyperbolic. This means that $\cN^s$ contracts volume and $\cN^u$ expand volume.

Since, for any $x_n$ in $\gamma_n$, $X_n(x_n)$ does not contract or expand  at the period, then $\cN^s_L\oplus<X_n(x_n)>$ contracts volume and  $<X_n(x_n)>\oplus\cN^u_L$ expands volume uniformly at the period.
 Since $\gamma_n$ tends to the homoclinic loop $\Gamma$, their periods
must tend to infinity with $n$. For $n$ large enough,
and from the contraction of volume we have there exist some constants $\nu$ and $T$
$$\prod^{\llcorner T_{\gamma_n}/T\lrcorner-1}_{i=0}\det(D\phi^{iT}(x)\mid_{\cN^s_L\oplus<X_n(x_n)>})\leq e^{-\nu T_{\gamma_n}}\,,$$
where $T_{\gamma_n}$ is the period of $\gamma_n$.

Then for any $\gamma_n$, taking $T=1$, Pliss Lemma (remark \ref{plisspoint}) gives some point
$p_n \in \gamma_n$ satisfying

$$\frac{1}{k}\sum^{k/1}_{i=0}\log(\det(D\phi_n^{1}\mid_{D\phi^{i}(\cN^s_L\oplus<X_n(p_n)>)}))\leq -\nu \,.$$

Assume $p_n$ tends to $y\in\Gamma\cup\sigma$. One
can assume $\cN^s_L\oplus<X_n(xn)>\to E(y)$, and since $y$ is accumulated by Pliss points, again we have that:
$$\frac{1}{k}\sum^{k}_{i=0}\log(\det(D\phi^{1}\mid_{D\phi^{i}(E(y))}))\leq -\nu \,,$$

Now the Pliss Lemma again,  allows us to find $n_j\to\infty$ such that
$$\frac{1}{k}\sum^{k}_{i=0}\log(\det(D\phi^{1}\mid_{D\phi^{i+n_j}(E(y))}))\leq -\nu \,.$$

Since $\phi^{i+n_j}(y)$ tends to $\sigma$ as $n_j\to\infty$, we derive a subspace $E\subset T_{\sigma}M$ with
$dim (E)=n+1$ such that $$\frac{1}{k}\sum^{k/1}_{i=0}\log(\det(D\phi^{1}\mid_{D\phi^{i}(E)}))\leq -\nu \,.$$

This shows that $E$ contracts volume, and the proof is analogous for $F$
 \endproof

The following corollary is a consequence of Lemma  \ref{l.repcont}.
\begin{coro}\label{c.repcont}
There is Let $\cV$ be the set of $\cC^1$ vector fields 
 $X$ with a robustly  chain transitive  class $C$ and such that:
\begin{itemize} 
\item  Every singularity in $C$ is hyperbolic and 
$$T_{\sigma}M=E^{ss}\oplus E^c\oplus E^{uu}\,,$$ notes the stable escaping, the unstable escaping and the center spaces.
\item there is  with a  finest dominated splitting over $\widetilde{\La}(X,U)$ 
$$\cN_L=\cN^s\oplus\cN^1\oplus\dots\oplus\cN^{k}\oplus\cN^u\,.$$
where  $U$, an isolating neighborhood of $C$,
and  $n=dim( \cN^s)>dim(E^{ss})$ 
 \end{itemize}
There is an open and dense subset $\cU\subset \cV$  such that  for every $X\in \cU$ every singularity in $C$ is such that there is a $n+1$ dimensional space $E\subset T_{\sigma}M$ that contracts volume. Moreover $E^{ss}\oplus E^c\subset E$

 \end{coro}
\proof
Let a vector field $X$ be a vector field such that 
\begin{itemize} 
\item  Every singularity in $C$ is hyperbolic and 
$$T_{\sigma}M=E^{ss}\oplus E^c\oplus E^{uu}\,,$$ notes the stable escaping, the unstable escaping and the center spaces.
\item there is  with a  finest dominated splitting over $\widetilde{\La}(X,U)$ 
$$\cN_L=\cN^s\oplus\cN^1\oplus\dots\oplus\cN^{k}\oplus\cN^u\,.$$
where  $U$, an isolating neighborhood of $C$,
and  $n=dim( \cN^s)>dim(E^{ss})$ 
\item all periodic orbits in $C$ are hyperbolic
\end{itemize}
Note that this is a dense subset of  the vector fields  $X$ with a robustly  chain transitive  class $C$ .

 By the connecting Lemma \ref{l.contecting}   we can find  a vector field $Y$   that  is $\epsilon-C^1$ close to  $Y' $ 
and is equal to $X$ in a neighborhood of $\sigma$, such that
 there is
 $\Gamma =Orb(x)$  a homoclinic orbit
associated to $\sigma$. Now Theorem \ref{ConConLem} allow us to find a sequence of vector fields
\begin{itemize}
\item there exists a sequence of star vector fields $Y_n$
converging to $Y$ in the $C^1$ topology
\item there exist a sequence of  periodic orbit $\gamma_n$ of $Y_n$  such
that $\gamma_n$ converges to $\Gamma$ in the Hausdorff topology.
\item there is a sequence of points $q_n\in\gamma_n$ such that $Y_n(q_n)\to u $ is in  $E^{u1}\oplus\dots\oplus E^{ul}$.
\item we call $L_u=<u>$ and we have that $dim(\cN^u_{L_u})=h$

 \item There is a sequence of points $p_n\in\gamma_n$ such that $<Y_n(p_n)>\to v $ is in  $E^{s1}\oplus\dots\oplus E^{sn}$
 \item we call $L_v=<v>$
and  we have that $dim(\cN^s_{L_v})=n$.
\end{itemize}
Since $Y$ is now in the conditions of lemma \ref{l.repcont} then there is an invariant space $E$ of dimension $n+1$ that contracts volume.
Since from corollary \ref{l.central} $dim(E^{ss}\oplus E^c)\leq n+1$, then  $E^{ss}\oplus E^c\subset E$.

Since in a neighborhood of a singularity $X$ and $Y$ are equal, and since the volume contraction of a subspace of $T_{\sigma}M$ is an open property, we get our result.
\endproof

As a direct consequence we have:
\begin{coro}
Let $\sigma$ be a singularity of $C$, a robustly  chain transitive  class with all singularities hyperbolic, with a finest dominated splitting over $\widetilde{\La}$
$$\cN_L=\cN^s\oplus\cN^1\oplus\dots\oplus\cN^{k}\oplus\cN^u\,.$$
Suppose that $n=dim( \cN^s)>dim(E^{ss})$ then $E^{cs}=E^{ss}\oplus E_{\sigma}^c$ contracts volume.

 \end{coro}
 
 The following corollary summarizes the situation 
 
\begin{coro}
Let $\sigma $ be a singularity in our chain recurrent class. We define $E^{cs}=E^{ss}\oplus E_{\sigma}^c$. We have 2 possibilities, 
\begin{itemize}
 \item either $\pi_L ( E^{c}_{\sigma})\cap\cN_L^s=\emptyset\,,$ and then there is a space $\cN^{ss}$ over the singularity contracts uniformly or
 \item $E^{cs}\subset \cN^{ss}\oplus L\subset E$ and $E$ contracts volume for $L\in E^c_{\sigma}$.
\end{itemize}
 \end{coro}

\section{Proof of the main theorem}

%%%%%%%%%%%%%%%%%%%%%%%%%%%%%%%%%%%%%%%%%%%%%%%%%%%%%%%%%%%%%%%%%%%%%%%%%%%%%%%%%%%%%%%%%%%%%%%%%%%%%%%%%%%%%%%%%%%%%%%%%%%%%%%%%%%%%%%%%%%%%%%%%%%%%%%%%%%%%%%%%%%%%%%%%%%%%%%%%%%%%%%%%%%%%%%%%%%%%%%%%%%%%%%%%%%%%%%%%%%%%%%%%%%%%%%%%%%%%%%%%%%%%%
We aim now to prove Theorem \ref{t.BDP}.
The proof is very similar to the proof in \cite{Ma2}. In fact is an adaptation to flows of the proof of theorem 4 in \cite{BDP}. The idea is to argue by contradiction and show that if there is no uniform volume expansion in the extremal bundle, then there is a closed orbit orbit of a sufficiently close vector field that contracts volume in the extremal bundle. This could be a periodic orbit or a singularity, but sections \ref{ss.centerspace} and \ref{s.per} show us that this is not possible.

The following proposition is equivalent to lemma 6.5  form  \cite{BDP}, and the proof is analogous.

\begin{lemm}\label{l.med}
Let $X\in\cX^1M $ be a vector field, $C$ a maximal invariant in a filtrating neighborhood $U\subset M $ %and  set $\widetilde{\La}$ .
Suppose there is a dominated splitting $E\oplus_{\prec} F$ over $\widetilde{\La}(X,U)$ for the reparametrized linear Poincar\'{e} flow, $h^T_{Ec}.\psi^T (L)$. Then if the Jacovian of $h^T_{Ec}.\psi^T (L)$  restricted to $E$ is not bounded from above by one, then for every $T$ there is a $h^T_{Ec}.\psi^T (L)$  invariant measure $\nu$ such that
$$\int log\abs{J(h^T_{Ec}.\psi^T (L) ,E)}\,d\nu\,\geq0\,.$$

\end{lemm}

Now we want to show that if  $h^T_{Ec}.\psi^T (L)$  does not contract volume on the most dominated bundle of the finest dominated splitting in  $\widetilde{\La}(X,U)$, then, the measure $\nu$ from the previous lemma is not supported on the directions that are over the singularities.

The following lemma is a consequence of corollary \ref{c.repcont}.
\begin{lemm}\label{l.sing}
Let $\cV\subset\cX^1M $ be the set of vector fields $X$ such that
\begin{itemize}
\item it has a robustly chain transitive class  $C$ that is  maximal invariant in a filtrating neighborhood $U\subset M $ with a singularity $\sigma$% and the set $\widetilde{\La}(X,U)$ .
\item there is a finest dominated splitting $$\cN_L=\cN^s_L\oplus_{\prec}\dots\oplus_{\prec}\cN^u_L$$ over $\widetilde{\La}(X,U)$  for the reparametrized linear Poincar\'{e} flow, $h^T_{Ec}.\psi^T (L)$
\end{itemize}
Them there is an open and dense subset $\cU\subset\cV$ such that for any $X\in\cU$ any $h^T_{Ec}.\psi^T (L)_T$ invariant measure $\nu$ supported in $\PP^c_{\sigma}\cap\widetilde{\La}(X,U)$ is such that
$$\int log\abs{J(h^T_{Ec}.\psi^T (L)_T,\cN^s_L)}\,d\nu\,<0\,.$$
\end{lemm}

\proof
Let us start by supposing that $\sigma\in S_E$ and  $dim(\cN^s)\leq E^{ss}$. In this case, we can include $\cN^s\subset T\sigma M$ as a subspace of $E^{ss}$. Then $\cN^s$ contracts uniformly for the tangent space and for the extended linear Poincar\'{e} flow. Note that in this case the reparametrized linear Poincar\'{e} flow and the extended linear Poincar\'{e} flow are equal in restriction to $dim(\cN^s)$ at  the directions over $\sigma$.

Suppose that $\sigma\in S_{Ec}$ and $n_s=dim(\cN^s)\geq E^{ss}$, then  given $L\in  \PP^c_{\sigma}\cap\widetilde{\La}(X,U)$ there exist a subspace $E=\cN^s\oplus L$ of $T_{\sigma}M$  that contracts volume, %corollary \ref{c.repcont} implies that $E^{cs}$ contracts volume.
The reparametrized linear Poincar\'{e} flow at the directions over $\sigma$ is $h^T_{Ec}.\psi^T (L)$ where $$h^T_{Ec}=\left(\frac{\norm{d\phi^t(u)}}{\norm{u}}\right)^{\frac{1}{n_s}}$$ for a non vanishing vector $u$ in the direction of $L$. Then $$\abs{J(\psi_T,\cN^s_L)}\left(\frac{\norm{d\phi^t(u)}}{\norm{u}}\right)=\abs{J(d\phi_T,E)}\,.$$
In any case lemma \ref{l.med} allows us to conclude.
\endproof

The flowing lemma is the only missing piece for Theorem \ref{t.BDP} for $\widetilde{\La}(X,U)$. Untill now we have from lemma \ref{l.med} that if there is no volume contraction of $\cN_L^S$, then there is a measure showing this lack of contraction.
From lemma \ref{l.sing} we also know that this measure can not be supported over a singularity.
Finally the next lemma uses the ergodic closing lemma to prove that if a measure was showing the lack of contraction, then it would be supported on a singularity contradicting the previous lemma.
So by contradiction the following lemma implies the volume contraction of the least dominated bundle. For the volume expansion the proof is analogus. Later we extend this structure over $\widetilde{\La}(X,U)$ to $B(C)$ and conclude with the proof of Theorem  \ref{t.BDP}.

\begin{lemm}\label{ultimo}
Let $X\in\cX^1M $ be a vector field, $L_a$ a maximal invariant in a filtrating neighborhood $U\subset M $ and the set $\widetilde{\La}$ .
Suppose there is a finest dominated splitting $$\cN_L=\cN^s_L\oplus_{\prec}\dots\oplus_{\prec}\cN^u_L$$ over $\widetilde{\La}(X,U)$ for the reparametrized linear Poincar\'{e} flow, $h^T_{Ec}.\psi^T (L)$. If there is a $h^T_{Ec}.\psi^T (L)$ invariant measure $\nu$ such that % not supported in  $\PP^c_{\sigma}\widetilde{\La}(X,U)$ then if
$$\int log\abs{J(h^T_{Ec}.\psi^T (L),\cN^s_L)}\,d\nu\,\geq 0\,,$$
then the measure must be supported on $\PP^c_{\sigma}\cap\widetilde{\La}(X,U)$ .
\end{lemm}
\proof
\begin{clai*} Let  $\nu_n$ be a measure supported on a periodic orbits $\gamma_n$ with period $\pi \gamma_n$ bigger than $T$ , then $\int\log h^{n_s.T}_{Ec} d\nu_n(x)=0$.

\end{clai*}
\begin{proof}

By definition of $h^T_{Ec}$
$$\log h^{n_s.T}_{Ec}\,d\nu_n(x)=\log\Pi_{\sigma_i\in S_{Ec}}\norm{h^{n_s \frac{1}{n_s}\,T}_{\sigma_i}}\,d\nu_n(x)\,,$$
so it suffices to prove the claim for a given $h^T_{\sigma_i}$.
For every $x$ in $\gamma$, since $h^T_{\sigma_i}$ is a multiplicative cocycle we have that:

\begin{eqnarray*}
 \Pi^{(m\pi(\gamma)/T)-1}_{i=0}&&h^T_{\sigma_i}(\phi^Y_{iT}(x))=h^{(m\pi(\gamma)/T)-1}_{\sigma_i}(x)
\end{eqnarray*}

The norm of the vector field restricted to $\gamma$ is bounded, and therefore $h^{(m\pi(\gamma)/T)-1}_{\sigma_i}(x)$ is bounded for $m\in\NN$ going to infinity.
Then this is also true for $h^T_{Ec}$.
Since $\nu_n$ is an ergodic measure, we have that
\begin{eqnarray*}
\int\log h^{T}_{\sigma_i}(x) d\nu_n(x)&=&\lim_{m\to\infty}\frac{1}{m}\sum^{(m\pi(\gamma)/T)-1}_{i=0}\log \left(h^{T}_{\sigma_i}(\phi^Y_{iT}(x))\right)\\
   &=&\lim_{m\to\infty}\frac{1}{m}\log \left(\Pi^{(m\pi(\gamma)/T)-1}_{i=0} h^T_{\sigma_i}(\phi^Y_{iT}(x))\right)\\
   &=&\lim_{m\to\infty}\frac{1}{m}\log\left(h^{(m\pi(\gamma)/T)-1}_{\sigma_i}(x)\right)\\
   &=&0
\end{eqnarray*}
\end{proof}

Suppose that  $\mu$ weights $0$ on $$\bigcup_{\sigma_i\in Sing(X)}\PP^c_{\sigma_i}\cap\widetilde{\La}(X,U)$$ 
then $\mu$ projects on $M$ on an ergodic measure $\nu$ supported on the class $C$ and such that ut weights $0$ in the singularities, for which
$$\int \log\abs{J(h_E.\psi^T_{\cN},\cN^s)} d\nu(x)\geq 0 .$$

Recall that $\psi^T $ is the linear Poincar\'{e} flow, and $h^T_{Ec}$ can be defined as a function of $x\in M$ instead of as a function of $L\in\PP M$ outside of an arbitrarily small neighborhood of the singularities.

 However,  the ergodic closing lemma implies that  $\nu$ is the weak$*$-limit of measures $\nu_n$ supported on closed orbits $\gamma_n$
 which converge for the Hausdorff distance to the support of $\nu$. Therefore, for $n$ large enough, the $\gamma_n$ are contained in
$C$ and from remark \ref{r.percontr} and our previous claim we know that
\begin{eqnarray*}
\int\log\abs{J(h^T_{Ec}.\psi^T_{\cN},\cN^s)} d\nu_n(x)) &=&\int\log\abs{J(\psi^T_{\cN},\cN^s)} d\nu_n(x) \\
\int\log\abs{J(\psi^T_{\cN},\cN^s)} d\nu_n(x) &\leq &  -\eta.
\end{eqnarray*}

Then $$\int\log\abs{J(h^T_{Ec}.\psi^T_{\cN},\cN^s)} d\nu_n(x)) \leq   -\eta$$
The map $\log\abs{J(h_{Ec}.\psi^T_{\cN},\cN^s)}$ is not continuous. Nevertheless, it is uniformly bounded and the unique discontinuity points are
the singularities of $X$. These singularities have (by assumption) weight $0$ for $\nu$ and thus admit neighborhoods with arbitrarily small weight.
Out of such a neighborhood the map is continuous.  One deduces that
$$\int\log\abs{J(h_{Ec}.\psi^T_{\cN},\cN^s)} d\nu(x)=\lim\int\log\abs{J(h_{Ec}.\psi^T_{\cN},\cN^s)} d\nu_n(x)$$
and therefore is strictly negative, contradicting the assumption.
\endproof

Note that all of this is also valid for the reverse time of the flow and for $\cN^{uu}$.

Lemmas \ref{l.sing} and \ref{ultimo} and their versions for the reverse time implies Theorem   \ref{t.BDP} over $\widetilde{\La}(X,U)$ . We re state it as the following corollary
\begin{coro}\label{coro1}
Let $\cV\subset\cX^1M $ be the set of vector fields $X$ such that
 it has a robustly chain transitive class  $C$ that is  maximal invariant in a filtrating neighborhood $U\subset M $ with a singularity $\sigma$% and the set $\widetilde{\La}(X,U)$ .

Them there is an open and dense subset $\cU\subset\cV$ such that for any $X\in\cU$ is singular volume partial hyperbolic over $\widetilde{\La}(X,U)$.
\end{coro}
Theorem 4 from \cite{BdL} gives immediately the following, which is equivalent to Theorem   \ref{t.BDP}.
\begin{coro}\label{coro2}
Let $\cV\subset\cX^1M $ be the set of vector fields $X$ such that
 it has a robustly chain transitive class  $C$ that is  maximal invariant in a filtrating neighborhood $U\subset M $ with a singularity $\sigma$% and the set $\widetilde{\La}(X,U)$ .

Them there is an open and dense subset $\cU\subset\cV$ such that for any $X\in\cU$ is singular volume partial hyperbolic over $B(C)$ .
\end{coro}
\begin{coro}
Let $X$ be a vector field with a robustly chain transitive attractor $C$ in the attracting set $U$ and all singularities in $C$ are hyperbolic. Then $C$ is singular volume partial hyperbolic. Moreover $C$ has a dominated splitting over $\widetilde{\La}(X,U)$ of the form $\cN_L=E\oplus F$ where  $E$ is contracting and $F$ is volume expanding for the reparametrized linear poincar\'e flow.
\end{coro}
\proof
Since $C$  is an attractor then $C$ has a dominated splitting over $\widetilde{\La}(X,U)$ of the form $\cN_L=E\oplus F$ where  $E$ is contracting. Given any singularity $\sigma$ in $C$ the center space $E^c$ projects to $F$ for all directions of  $\widetilde{\La}(X,U)$ over  $\sigma$. Therefore by Lemma \ref{l.central} the finest dominated splitting of $C$ has $F$ as the dominating extremal bundle. Therefore by Corollary \ref{coro2} $F$ has to be  volume expanding for the reparametrized linear Poincar\'e flow.
\endproof
\begin{rema}
Note that since the example in \cite{BLY} is in de conditions of corolary \ref{coro2} then it is singular volume partial hyperbolic.
\end{rema}

\end{document}